\documentclass[a4paper.12pt]{article}

\usepackage[dvips]{color}
\usepackage{graphicx}
\usepackage{amssymb,amsmath,amsthm}
\usepackage{xy}
\usepackage{wrapfig}
\usepackage{dsfont}
\usepackage{epsf}
\input xy
\xyoption{all}

\usepackage{latexsym, amsfonts, amssymb, txfonts, pxfonts, wasysym}
\usepackage{latexsym}
\usepackage{algorithm}
\usepackage{caption}

\theoremstyle{remark}

\theoremstyle{theorem}

\newtheorem{theorem}[subsection]{\hspace{7 mm}Theorem}

\renewcommand\proofname[1]{\hspace{7 mm}Proof}

\title{A GEOMETRIC WAY TO GENERATE BLUNDON TYPE INEQUALITIES}

\author{Dorin Andrica, Catalin Barbu and Nicusor Minculete}

\date{}

\begin{document}

\maketitle

\begin{abstract}
We present a geometric way to generate Blundon type inequalities.
Theorem 3.1 gives the formula for $\cos \widehat{POQ}$ in terms of
the barycentric coordinates of the points $P$ and $Q$ with respect
to a given triangle. This formula implies Blundon type inequalities
generated by the points $P$ and $Q$ (Theorem 3.2). Some applications
are given in the last section by choosing special points $P$ and
$Q$.

\hspace{0 mm} \textbf{2000 Mathematics Subject Classification}:
26D05; 26D15; 51N35.

\end{abstract}

\section{Introduction}

Consider $O$ the circumcenter, $I$ the incenter , $ G $ the
centroid, $N$ the Nagel point, $s$ the semiperimeter, $R$ the
circumradius, and $r$ the inradius of triangle $ABC$.

Blundon's inequalities express the necessary and sufficient
conditions for the existence of a triangle with elements $s$, $R$
and $r$:

\begin{equation}
2R^{2}+10Rr-r^{2}-2(R-2r)\sqrt{R^{2}-2Rr}\leq s^{2}\leq
2R^{2}+10Rr-r^{2}+2(R-2r)\sqrt{R^{2}-2Rr}. \tag{1}
\end{equation}

Clearly these two inequalities can be written in the following
equivalent form

\begin{equation}
|s^{2} -2R^{2}-10Rr+r^{2}| \leq 2(R-2r)\sqrt{R^{2}-2Rr}, \tag{2}
\end{equation}
and in many references this relation is called the fundamental
inequality of triangle $ABC$.

The standard proof is an algebraic one, it was first time given by
W.J.Blundon [5] and it is based on the characterization of cubic
equations with the roots the length sides of a triangle. For more
details we refer to the monograph of D. Mitrinovi\'{c}, J.
Pe\v{c}ari\'{c}, V. Volenec [16], and to the papers of C.Niculescu
[17],[18]. R.A.Satnoianu [20], and S.Wu [22] have obtained some
improvements of this important inequality.

The following result was obtained by D.Andrica and C.Barbu in the
paper [3] and it contains a simple geometric proof of (1). Assume
that the triangle $ABC$ is not equilateral. The following relation
holds :
\begin{equation}
\cos \widehat{ION}=\frac{2R^2 +10Rr -r^2 -s^2}{2(R-2r)\sqrt{R^2
-2Rr}}. \tag{3}
\end{equation}

If we have $R=2r$, then the triangle must be equilateral and we have
equality in (1) and (2). If we assume that $R-2r\neq 0$, then
inequalities (1) are direct consequences of the fact that $-1 \leq
\cos \widehat{ION}\leq 1$.

In this geometric argument the main idea is to consider the points
$O$, $I$ and $N$, and then to get the formula (3). It is a natural
question to see what is a similar formula when we kip the
circumcenter $O$ and we replace the points $I$ and $N$ by other two
points $P$ and $Q$. In this way we obtain Blundon type inequalities
generated by the points $P$ and $Q$. Section 2 contains the basic
facts about the main ingredient helping us to do all the
computations, that is the barycentric coordinates. In Section 3 we
present the analogous formula to (3), for the triangle $POQ$, and
the we derive the Blundon type inequalities generated in this way.
The last section contains some applications of the results in
Section 3 as follows: the classical Blundon's inequalities, the dual
Blundon's inequalities obtained in the paper [3], the Blundon's
inequalities generated by two Cevian points of rank $(k;l;m)$.

\section{Some basic results about barycentric coordinates}

Let $P$ be a point situated in the plane of the triangle $ABC$. The
Cevian triangle $DEF$ is defined by the intersection of the Cevian
lines though the point $P$ and the sides $BC, CA, AB$ of triangle.
If the point $P$ has barycentric coordinates $t_{1}:t_{2}:t_{3}$,
then the vertices of the Cevian triangle $DEF$ have
barycentric coordinates given by:\ $%
D(0:t_{2}:t_{3}),E(t_{1}:0:t_{3})$ and $F(t_{1}:t_{2}:0).$ The
barycentric coordinates were introduced in 1827 by M\"{o}bius (see
[10]). The using of barycentric coordinates defines a distinct part
of Geometry called Barycentric Geometry. More details can be found
in the monographs of C. Bradley [10], C. Coand\u{a}[11], C.
Co\c{s}ni\c{t}\u{a} [12], C. Kimberling [14], and in the papers of
O. Bottema [9], J. Scott [21], and P. Yiu [23].

It is well-known ([11],[12]) that for every point $M$ in the plane
of triangle $ABC$, then the following relation holds:
\begin{equation}
(t_{1}+t_{2}+t_{3})\overrightarrow{MP}=t_{1}\overrightarrow{MA}+t_{2}
\overrightarrow{MB}+t_{3}\overrightarrow{MC}.  \tag{4}
\end{equation}
In the particular case when $M\equiv P,$ we obtain
\[
t_{1}\overrightarrow{PA}+t_{2}\overrightarrow{PB}+t_{3}\overrightarrow{PC}=%
\overrightarrow{0}.
\]
This last relation shows that the point $P$ is the barycenter of the
system $\{A,B,C\}$ with the weights $\{t_1,t_2,t_3\}$. The following
well-known result is very useful in computing distances from the
point $M$ to the barycenter $P$ of the system $\{A,B,C\}$ with the
weights $\{t_1,t_2,t_3\}$.

\begin{theorem}
If $M$ is a point situated in the plane of triangle $ABC$, then
\begin{equation}
(t_{1}+t_{2}+t_{3})^{2}MP^{2}=(t_{1}MA^{2}+t_{2}MB^{2}+t_{3}MC^{2})(t_{1}+t_{2}+t_{3})-(t_{2}t_{3}a^{2}+t_{3}t_{1}b^{2}+t_{1}t_{2}c^{2}),
\tag{5}
\end{equation}
where $a=BC, b=CA, c=AB$, are the length sides of triangle.
\end{theorem}

\begin{proof}
Using the scalar product of two vectors, from (4) we obtain:
\[
(t_{1}+t_{2}+t_{3})^{2}MP^{2}=t_{1}^{2}MA^{2}+t_{2}^{2}MB^{2}+t_{3}^{2}MC^{2}+
2t_{1}t_{2}\overrightarrow{MA}\cdot \overrightarrow{MB}+2t_{1}t_{3}
\overrightarrow{MA}\cdot
\overrightarrow{MC}+2t_{2}t_{3}\overrightarrow{MB} \cdot
\overrightarrow{MC},
\]
that is
\[
(t_{1}+t_{2}+t_{3})^{2}MP^{2}=t_{1}^{2}MA^{2}+t_{2}^{2}MB^{2}+t_{3}^{2}MC^{2}+
\]
\[
t_{1}t_{2}(MA^{2}+MB^{2}-AB^{2})+t_{1}t_{3}(MA^{2}+MC^{2}-AC^{2})+t_{2}t_{3}(MB^{2}+MC^{2}-BC^{2}),
\]
hence,
\[
(t_{1}+t_{2}+t_{3})^{2}MP^{2}=(t_{1}MA^{2}+t_{2}MB^{2}+t_{3}MC^{2})(t_{1}+t_{2}+t_{3})-(t_{2}t_{3}a^{2}+t_{3}t_{1}b^{2}+t_{1}t_{2}c^{2}).
\]
To get the last relation we have used the definition of the scalar
product and the Cosine Law as follows
\[
2\overrightarrow{MA}\cdot \overrightarrow{MB}=2MA\cdot MB\cos
\widehat{AMB}=2MA\cdot MB\cdot \frac{MA^{2}+MB^{2}-AB^{2}}{2MA\cdot
MB}=MA^{2}+MB^{2}-AB^{2}.
\]
\end{proof}

If we consider that $t_{1},t_{2},t_{3},$ and $t_{1}+t_{2}+t_{3}$ are
nonzero real numbers, then the relation $(5)$ becomes the Lagrange's
relation
\begin{equation}
MP^{2}=\frac{t_{1}MA^{2}+t_{2}MB^{2}+t_{3}MC^{2}}{t_{1}+t_{2}+t_{3}}-\frac{
t_{1}t_{2}t_{3}}{(t_{1}+t_{2}+t_{3})^{2}}\left(
\frac{a^{2}}{t_{1}}+\frac{ b^{2}}{t_{2}}+\frac{c^{2}}{t_{3}}\right)
.  \tag{6}
\end{equation}

If we consider in (6) $M\equiv O$, the circumcenter of the triangle,
then it follows
\begin{equation}
R^{2}-OP^{2}=\frac{t_{1}t_{2}t_{3}}{(t_{1}+t_{2}+t_{3})^{2}}\left(
\frac{ a^{2}}{t_{1}}+\frac{b^{2}}{t_{2}}+\frac{c^{2}}{t_{3}}\right).
\tag{7}
\end{equation}

The following version of Cauchy-Schwarz inequality is also known in
the literature as Bergstr\"{o}m's inequality (see [6], [7], [8],
[19]): If $x_{k}, a_k \in \mathbb{R}$ and $a_{k}>0, k=1,2,\cdots
,n,$ then
\[
\frac{x_{1}^{2}}{a_{1}}+\frac{x_{2}^{2}}{a_{2}}+...+\frac{x_{n}^{2}}{a_{n}}
\geq \frac{(x_{1}+x_{2}+...+x_{n})^{2}}{a_{1}+a_{2}+...+a_{n}},
\]
with equality if and only if
\[
\frac{x_{1}}{a_{1}}=\frac{x_{2}}{a_{2}}=...=\frac{x_{n}}{a_{n}}.
\]

Using Bergstr\"{o}m's inequality and relation (4), we obtain
\[
R^{2}-OP^{2}\geq
\frac{t_{1}t_{2}t_{3}}{(t_{1}+t_{2}+t_{3})^{2}}\cdot \frac{
(a+b+c)^{2}}{t_{1}+t_{2}+t_{3}},
\] that is in any triangle with semiperimeter $s$ the following inequality holds:
\[
R^{2}-OP^{2}\geq
\frac{4s^{2}t_{1}t_{2}t_{3}}{(t_{1}+t_{2}+t_{3})^{3}},
\]
where $t_{1}:t_{2}:t_{3}$ are the barycentric coordinates of $P$ and
$t_{1},t_{2},t_{3}>0.$ Equality holds if an only if
$t_{1}=a,t_{2}=b, t_{3}=c$, that is $P\equiv I$, the incenter of the
triangle $ABC$.

\begin{theorem}
$([11],[12]).$ If the points $P$ and $Q$ have barycentric
coordinates $ t_{1}:t_{2}:t_{3}$, and $u_{1}:u_{2}:u_{3}$,
respectively, with respect to the triangle $ABC,$ and
$u=u_{1}+u_{2}+u_{3},t=t_{1}+t_{2}+t_{3}$, then
\begin{equation}
PQ^{2}=-\alpha \beta \gamma \left( \frac{a^{2}}{\alpha
}+\frac{b^{2}}{\beta } +\frac{c^{2}}{\gamma }\right),  \tag{8}
\end{equation}
where the numbers $\alpha, \beta, \gamma$ are defined by
\[
\alpha =\frac{u_{1}}{u}-\frac{t_{1}}{t} ;\beta
=\frac{u_{2}}{u}-\frac{t_{2}}{t} ;\gamma
=\frac{u_{3}}{u}-\frac{t_{3}}{t}.
\]
\end{theorem}

\section{Blundon type inequalities generated by two points}

\begin{theorem}
Let $P$ and $Q$ be two points different from the circumcircle $O$,
having the barycentric coordinates $t_{1}:t_{2}:t_{3}$ and
$u_{1}:u_{2}:u_{3}$ with respect to the triangle $ABC,$ and let
$u=u_{1}+u_{2}+u_{3},t=t_{1}+t_{2}+t_{3}$. If $t_{1},t_{2},t_{3},
u_{1},u_{2},u_{3}\neq 0$, then the following relation holds
\begin{equation}
\cos \widehat{POQ}=\frac{2R^{2}-\frac{t_{1}t_{2}t_{3}}{ t^{2}}\left(
\frac{a^{2}}{t_{1}}+\frac{b^{2}}{t_{2}}+\frac{ c^{2}}{t_{3}}\right)
-\frac{u_{1}u_{2}u_{3}}{u^{2}}\left(
\frac{a^{2}}{u_{1}}+\frac{b^{2}}{u_{2}}+\frac{c^{2}}{u_{3}}\right)
+\alpha \beta \gamma \left( \frac{a^{2}}{\alpha }+\frac{b^{2}}{\beta
}+\frac{c^{2}}{ \gamma }\right) }{2\sqrt{\left[
R^{2}-\frac{t_{1}t_{2}t_{3}}{ t^2}\left(
\frac{a^{2}}{t_{1}}+\frac{b^{2}}{t_{2}}+\frac{ c^{2}}{t_{3}}\right)
\right] \cdot \left[ R^{2}-\frac{u_{1}u_{2}u_{3}}{ u^{2}}\left(
\frac{a^{2}}{u_{1}}+\frac{b^{2}}{u_{2}}+\frac{ c^{2}}{u_{3}}\right)
\right] }}  \tag{9}
\end{equation}
where $a,b,c$ are the length sides of the triangle and
\[
\alpha =\frac{u_{1}}{u}-\frac{t_{1}}{t} ;\beta
=\frac{u_{2}}{u}-\frac{t_{2}}{t} ;\gamma
=\frac{u_{3}}{u}-\frac{t_{3}}{t}.  \tag{10}
\]

\end{theorem}

\begin{proof}
Applying the relation (7) for the points $P$ and $Q$, we have
\begin{equation}
OP^{2}=R^{2}-\frac{t_{1}t_{2}t_{3}}{t^{2}}\left( \frac{
a^{2}}{t_{1}}+\frac{b^{2}}{t_{2}}+\frac{c^{2}}{t_{3}}\right)
\tag{11}
\end{equation}
and%
\begin{equation}
OQ^{2}=R^{2}-\frac{u_{1}u_{2}u_{3}}{u^{2}}\left( \frac{
a^{2}}{u_{1}}+\frac{b^{2}}{u_{2}}+\frac{c^{2}}{u_{3}}\right)
\tag{12}
\end{equation}
We use the Law of Cosines in the triangle $POQ$ to obtain
\begin{equation}
\cos \widehat{POQ}=\frac{OP^{2}+OQ^{2}-PQ^{2}}{2OP\cdot OQ},
\tag{13}
\end{equation}
and from relations (8), (11), (12) and (13) we obtain the relation
(9).
\end{proof}

\begin{theorem}
Let $P$ and $Q$ be two points different from the circumcircle $O$,
having the barycentric coordinates $t_{1}:t_{2}:t_{3}$ and
$u_{1}:u_{2}:u_{3}$ with respect to the triangle $ABC,$ and let
$u=u_{1}+u_{2}+u_{3},t=t_{1}+t_{2}+t_{3}$. If $t_{1},t_{2},t_{3},
u_{1},u_{2},u_{3}\neq 0$, then the following inequalities hold
\[
-2\sqrt{ \left[ R^{2}-\frac{t_1t_2t_3}{t^2}\left(
\frac{a^2}{t_1}+\frac{b^2}{t_2}+\frac{c^2}{t_3}\right) \right] \cdot
\left[ R^{2}-\frac{u_1u_2u_3}{u^2}\left(
\frac{a^2}{u_1}+\frac{b^2}{u_2}+\frac{c^2}{u_3}\right) \right] }
\leq
\]
\[
\alpha \beta \gamma \left( \frac{a^2}{\alpha }+\frac{b^2}{\beta
}+\frac{ c^2}{\gamma }\right) +2R^{2}-\left[ \frac{t_1t_2t_3}{
t^2}\left( \frac{a^2}{t_1}+\frac{b^2}{t_2}+\frac{ c^2}{t_3}\right)
+\frac{u_1u_2u_3}{u^2}\left(
\frac{a^2}{u_1}+\frac{b^2}{u_2}+\frac{c^2}{u_3}\right) \right] \leq
\]
\begin{equation}
2\sqrt{ \left[ R^{2}-\frac{t_1t_2t_3}{t^2}\left(
\frac{a^2}{t_1}+\frac{b^2}{t_2}+\frac{c^2}{t_3}\right) \right] \cdot
\left[ R^{2}-\frac{u_1u_2u_3}{u^2}\left(
\frac{a^2}{u_1}+\frac{b^2}{u_2}+\frac{c^2}{u_3}\right) \right] }
\tag{14}
\end{equation}
where $a,b,c$ are the length sides of the triangle and the numbers
$\alpha, \beta, \gamma$ are defined by (10).

\end{theorem}
\begin{proof}
The inequalities (14) are simple direct consequences of the fact
that $ -1\leq \cos \widehat{POQ}\leq 1.$ The equality in the right
inequality holds if and only if $\widehat{POQ}=0$, that is the
points $O, P, Q$ are collinear in the order $O, P, Q$ or $O, Q, P$.
The equality in the left inequality holds if and only if
$\widehat{POQ}=\pi$, that is the points $O, P, Q$ are collinear in
the order $P, O, Q$ or $Q, O, P$.
\end{proof}

From Theorem 3.1 it follows that it is a natural and important
problem to construct the triangle $ABC$ from the points $O, P, Q$,
when we know their barycentric coordinates. In the special case when
$P\equiv I$ and $Q\equiv N$ we know that that points $I,G,N$ are
collinear, determining the Nagel line of triangle, and the centroid
$G$ lies on the segment $IN$ such that $IG=\frac{1}{3}IN$. Then,
using the Euler's line of the triangle, we get the orthocenter $H$
on the ray $(OG$ such that $OH=3OG$. In this case the problem is
reduced to the famous Euler's determination problem i.e. to
construct a triangle from its incenter $I$, circumcenter $O$, and
orthocenter $H$ (see the paper of P.Yiu [24] for details and
results). This is a reason to call the problem as the general
determination problem.

\section{Applications}

The formula (3) and the classical Blundon's inequalities (1) can be
obtained from (9) and (14) by considering $P=I$, the incenter, and
$Q=N$, the Nagel point of the triangle. Indeed, the barycentric
coordinates of incenter $I\ $ and of Nagel's point $N\ $ are
$(t_{1},t_{2},t_{3})=(a,b,c)$, and
$(u_{1},u_{2},u_{3})=(s-a,s-b,s-c),$ respectively. We have
\begin{equation}
u=u_{1}+u_{2}+u_{3}=s-a+s-b+s-c=s,
u_{1}u_{2}u_{3}=(s-a)(s-b)(s-c)=r^{2}s,  \tag{15}
\end{equation}
and
\begin{equation}
t=t_{1}+t_{2}+t_{3}=2s,\text{\ }t_{1}t_{2}t_{3}=abc=4Rrs.  \tag{16}
\end{equation}
We obtain
\begin{equation}
\alpha =\frac{s-a}{s}-\frac{a}{2s}=\frac{2s-3a}{2s}, \beta =\frac{
2s-3b}{2s}, \gamma =\frac{2s-3c}{2s}.  \tag{17}
\end{equation}
Therefore
\[
\alpha \beta \gamma \left( \frac{a^{2}}{\alpha }+\frac{b^{2}}{\beta
}+\frac{ c^{2}}{\gamma }\right) =\sum_{cyc} \beta \gamma
a^{2}=\sum_{cyc} \left( 1-\frac{3b}{2s} \right) \left(
1-\frac{3c}{2s}\right) a^{2}=
\]
\[
\sum{cyc} a^{2}-\frac{3}{2s}\sum_{cyc} \left[
a^{2}(a+b+c)-a^{3}\right] +\frac{9abc}{ 4s^{2}}\sum_{cyc} a=
\]
\[
\sum_{cyc} a^{2}-3\sum_{cyc} a^{2}+\frac{3}{2s}\sum_{cyc}
a^{3}+\frac{9abc}{2s}=
\]
\[
-2(2s^{2}-2r^{2}-8Rr)+3(s^{2}-3r^{2}-6Rr)+18Rr
\]
that is
\begin{equation}
\alpha \beta \gamma \left( \frac{a^2}{\alpha }+\frac{b^2}{\beta
}+\frac{c^2}{\gamma } \right) =-s^2-5r^2+16Rr.  \tag{18}
\end{equation}
Now, using (16) and (17) we get
\begin{equation}
\frac{t_{1}t_{2}t_{3}}{t}\left( \frac{a^{2}}{t_{1}}+
\frac{b^{2}}{t_{2}}+\frac{c^{2}}{t_{3}}\right)
=\frac{4Rrs}{4s^2}\cdot 2s=2Rr,  \tag{19}
\end{equation}
and
\[
\frac{u_{1}u_{2}u_{3}}{u}\left( \frac{a^{2}}{u_{1}}+
\frac{b^{2}}{u_{2}}+\frac{c^{2}}{u_{3}}\right)
=\frac{r^{2}s}{s^{2}}\left(
\frac{a^{2}}{s-a}+\frac{b^{2}}{s-b}+\frac{c^{2}}{s-c}\right) =
\]
\[
\frac{r^{2}}{s}\cdot \frac{\sum_{cyc} a^{2}(s-b)(s-c)}{r^{2}s}=\frac{1}{s^{2}}%
\left\{ s^{2}\sum_{cyc} a^{2}-s\sum_{cyc} \left[
a^{2}(a+b+c)-a^{3}\right] +abc\sum_{cyc} a\right\} =
\]
\[
\frac{1}{s^{2}}\left( s\sum_{cyc} a^{3}-s^{2}\sum_{cyc}
a^{2}+8Rrs^{2}\right) =
\]
\begin{equation}
\frac{1}{s^{2}}\left[
2s^{2}(s^{2}-3r^{2}-6Rr)-s^{2}(2s^{2}-2r^{2}-8Rr)+8Rrs^{2}\right]
=4Rr-4r^{2} \tag{20}
\end{equation}
Using the relations (18)-(20) in (9) we obtain the relation (3).
These computations are similar to those given by complex numbers in
[1].

Now, consider the excenters $I_a ,I_b ,I_c$, and $N_a ,N_b ,N_c$ the
adjoint points to the Nagel point $N$. For the definition and some
properties of the adjoint points $N_a ,N_b ,N_c$ we refer to the
paper of D.Andrica and K.L.Nguyen [2]. Let $s, R, r, r_a ,r_b ,r_c$
be the semiperimeter, circumradius, inradius, and exradii of
triangle $ABC$, respectively. Considering the triangle $I_aON_a$,
D.Andrica and C.Barbu [3] have proved the following formula
\begin{equation}
\cos \widehat{I_aON_a}= \frac{R^2 -3Rr_a -r_a^2 -\alpha}{(R+2r_a
)\sqrt{R^2 +2Rr_a}},\tag{21}
\end{equation}
where $\alpha =\frac{a^2 +b^2 +c^2}{4}$.

Using formula (21),we get the dual form of Blundon's inequalities
given in the paper [3]
\begin{equation}
0\leq \frac{a^2 +b^2 +c^2}{4}\leq R^2 -3Rr_a -r_a^2 +(R+2r_a
)\sqrt{R^{2}+2Rr_a}. \tag{22}
\end{equation}
There are similar inequalities involving the exradii $r_b$ and
$r_c$.

We known that the barycentric coordinates of the excenter $I_{a}$ are $%
(t_{1},t_{2},t_{3})=(-a,b,c)$, and of the adjoint Nagel point $N_{a}$ are $%
(u_{1},u_{2},u_{3})=(s,c-s,b-s)$. Using formula (9) we can obtain
the relation (21) and then the dual form of the classical Blundon's
inequalities (22).

We have
\begin{equation*}
u=u_{1}+u_{2}+u_{3}=s-a,\text{ }u_{1}u_{2}u_{3}=s(s-b)(s-c)
\end{equation*}%
and
\begin{equation*}
t=t_{1}+t_{2}+t_{3}=2(s-a),\text{\ }t_{1}t_{2}t_{3}=-abc=-4Rrs.
\end{equation*}%
We obtain
\begin{equation*}
\alpha =\frac{2s+a}{2(s-a)}=1+\frac{3a}{2(s-a)},
\end{equation*}%
\begin{equation*}
\beta =\frac{2c-2s-b}{2(s-a)}=1-\frac{3b}{2(s-a)},
\end{equation*}%
\begin{equation*}
\gamma =\frac{2b-2s-c}{2(s-a)}=1-\frac{3c}{2(s-a)}.
\end{equation*}%
Therefore,
\begin{equation}
\frac{t_{1}t_{2}t_{3}}{t^{2}}\left( \frac{a^{2}}{t_{1}}+\frac{b^{2}}{t_{2}}+%
\frac{c^{2}}{t_{3}}\right) =\frac{-4Rrs}{4(s-a)^{2}}\cdot
2(s-a)=-2Rr_{a}, \tag{23}
\end{equation}%
and
\begin{equation*}
\frac{u_{1}u_{2}u_{3}}{u^{2}}\left( \frac{a^{2}}{u_{1}}+\frac{b^{2}}{u_{2}}+%
\frac{c^{2}}{u_{3}}\right) =\frac{s(s-c)(s-b)}{(s-a)^{2}}\left( \frac{a^{2}}{%
s}-\frac{b^{2}}{s-c}-\frac{c^{2}}{s-b}\right) =
\end{equation*}%
\begin{equation*}
-\left( -a^{2}\cdot \frac{s-b}{s-a}\cdot \frac{s-c}{s-a}+b^{2}\cdot \frac{s}{%
s-a}\cdot \frac{s-b}{s-a}+c^{2}\cdot \frac{s}{s-a}\cdot \frac{s-c}{s-a}%
\right) =
\end{equation*}%
\begin{equation*}
-\left( -a^{2}\cdot \frac{r_{a}}{r_{b}}\cdot
\frac{r_{a}}{r_{c}}+b^{2}\cdot \frac{r_{a}}{r}\cdot
\frac{r_{a}}{r_{b}}+c^{2}\cdot \frac{r_{a}}{r}\cdot
\frac{r_{a}}{r_{c}}\right) =
\end{equation*}%
\begin{equation}
-r_{a}^{2}\left( \frac{-a^{2}}{r_{b}r_{c}}+\frac{b^{2}}{rr_{b}}+\frac{c^{2}}{%
rr_{c}}\right) =-r_{a}^{2}\left( \frac{4R}{r_{a}}+4\right)
=-4Rr_{a}-4r_{a}^{2},  \tag{24}
\end{equation}%
where we have used the relation $\frac{-a^{2}}{r_{b}r_{c}}+\frac{b^{2}}{rr_{b}}+%
\frac{c^{2}}{rr_{c}}=\frac{4R}{r_{a}}+4$ (see [2], p. 134).

Now, we will calculate the expression:%
\begin{equation*}
E=\alpha \beta \gamma \left( \frac{a^{2}}{\alpha }+\frac{b^{2}}{\beta }+%
\frac{c^{2}}{\gamma }\right) +\frac{a^{2}+b^{2}+c^{2}}{2}=
\end{equation*}%
\begin{equation*}
a^{2}\beta \gamma +\frac{a^{2}}{2}+b^{2}\alpha \gamma +\frac{b^{2}}{2}%
+c^{2}\alpha \beta +\frac{c^{2}}{2}=
\end{equation*}%
\begin{equation*}
a^{2}\left[ 1-\frac{3(b+c)}{2(s-a)}+\frac{9bc}{4(s-a)^{2}}\right] +\frac{%
a^{2}}{2}+
\end{equation*}%
\begin{equation*}
b^{2}\left[ 1+\frac{3(a-c)}{2(s-a)}-\frac{9ca}{4(s-a)^{2}}\right] +\frac{%
b^{2}}{2}+
\end{equation*}%
\begin{equation*}
c^{2}\left[ 1+\frac{3(a-b)}{2(s-a)}-\frac{9ab}{4(s-a)^{2}}\right] +\frac{%
c^{2}}{2},
\end{equation*}%
that is %
\begin{equation*}
E=a^{2}\left[ \frac{-3s}{2(s-a)}+\frac{9bc}{4(s-a)^{2}}\right]
+b^{2}\left[
\frac{3(s-c)}{2(s-a)}-\frac{9ca}{4(s-a)^{2}}\right] +c^{2}\left[ \frac{3(s-b)%
}{2(s-a)}-\frac{9ab}{4(s-a)^{2}}\right] =
\end{equation*}%
\begin{equation*}
\frac{3}{2(s-a)}[-a^{2}s+b^{2}(s-c)+c^{2}(s-b)]+\frac{9abc}{4(s-a)^{2}}%
(a-b-c)=
\end{equation*}%
\begin{equation}
\frac{3}{2(s-a)}[s(-a^{2}+b^{2}+c^{2})-bc(b+c)]-18Rr_{a}.  \tag{25}
\end{equation}%
We have
\begin{equation*}
s(-a^{2}+b^{2}+c^{2})-bc(b+c)=2sbc\cos A-2bcs+abc=
\end{equation*}%
\begin{equation*}
2sbc(\cos A-1)+abc=abc-4sbc\sin ^{2}\frac{A}{2}=
\end{equation*}%
\begin{equation}
abc-4s(p-b)(p-c)=abc-4Sr_{a}=4S(R-r_{a}),  \tag{26}
\end{equation}%
where $S$ denotes the area of triangle $ABC.$ From relations (25) and (26) we get%
\begin{equation*}
E=\frac{3}{2(s-a)}\cdot
4S(R-r_{a})-18Rr_{a}=6r_{a}(R-r_{a})-18Rr_{a}=-12Rr_{a}-6r_{a}^{2},
\end{equation*}%
therefore%
\begin{equation}
\alpha \beta \gamma \left( \frac{a^{2}}{\alpha }+\frac{b^{2}}{\beta }+\frac{%
c^{2}}{\gamma }\right)
=-12Rr_{a}-6r_{a}^{2}-\frac{a^{2}+b^{2}+c^{2}}{2} \tag{27}
\end{equation}

Using formulas (23), (24) and (27) in the general formula (9) we
obtain the relation (21).

In the paper $[13],$ N. Minculete and C. Barbu have introduced the
Cevians
of rank $(k;l;m)$. The line $AD$ is called \textit{ex-Cevian of rank }$%
\mathit{(k;l;m)}$ \textit{\ }or \textit{exterior Cevian of rank }$\mathit{%
(k;l;m),}$ if the point $D$ is situated on side $\left( BC\right) $
of the non-isosceles triangle $ABC$ and the following relation
holds:
\begin{equation*}
\frac{BD}{DC}=\left( \frac{c}{b}\right) ^{k}\cdot \left( \frac{s-c}{s-b}%
\right) ^{l}\cdot \left( \frac{a+b}{a+c}\right) ^{m}.
\end{equation*}%
In the paper $[13]$ it is proved that the Cevians of rank $(k;l;m)$
are
concurrent in the point $I(k,l,m)$ called \textit{the Cevian point of rank }$%
\mathit{(k;l;m)}$ and the barycentric coordinates of $I(k,l,m)$ are:
\begin{equation*}
a^{k}(s-a)^{l}(b+c)^{m}:b^{k}(s-b)^{l}(a+c)^m:c^{k}(s-c)^{l}(a+b)^{m}.
\end{equation*}
In the case $l=m=0$, we obtain the Cevian point of rank $k$.

Let $I_{1},I_{2}$ be two Cevian points with barycentric coordinates:
\begin{equation*}
I_{i}[a^{k_{i}}(s-a)^{l_{i}}(b+c)^{m_{i}}:b^{k_{i}}(s-b)^{l_{i}}(a+c)^{m_{i}}:c^{k_{i}}(s-c)^{l_{i}}(a+b)^{m_{i}}],%
\text{ }i=1,2.
\end{equation*}%
Denote $t_i^1 = a^{k_{i}}(s-a)^{l_{i}}(b+c)^{m_{i}}$, $%
t_i^2=b^{k_{i}}(s-b)^{l_{i}}(a+c)^{m_{i}},t_i^3
=c^{k_{i}}(s-c)^{l_{i}}(a+b)^{m_{i}}, i=1,2.$  From formula (9) we
obtain
\begin{equation}
\cos \widehat{I_{1}OI_{2}}=\frac{2R^{2}-\frac{t_{1}^{1}t_{1}^{2}t_{1}^{3}}{%
(T_1)^{2}}\left( \frac{a^{2}}{t_{1}^{1}}+\frac{b^{2}}{t_{1}^{2}}+\frac{c^{2}}{%
t_{1}^{3}}\right) -\frac{t_{2}^{1}t_{2}^{2}t_{2}^{3}}{(T_2)^{2}}\left( \frac{%
a^{2}}{t_{2}^{1}}+\frac{b^{2}}{t_{2}^{2}}+\frac{c^{2}}{t_{2}^{3}}\right)
+\alpha \beta \gamma \left( \frac{a^{2}}{\alpha }+\frac{b^{2}}{\beta }+\frac{%
c^{2}}{\gamma }\right) }{2\sqrt{\left[ R^{2}-\frac{%
t_{1}^{1}t_{1}^{2}t_{1}^{3}}{(T_1)^{2}}\left( \frac{a^{2}}{t_{1}^{1}}+\frac{b^{2}%
}{t_{1}^{2}}+\frac{c^{2}}{t_{1}^{3}}\right) \right] \cdot \left[ R^{2}-\frac{%
t_{2}^{1}t_{2}^{2}t_{2}^{3}}{(T_2)^{2}}\left( \frac{a^{2}}{t_{2}^{1}}+\frac{b^{2}%
}{t_{2}^{2}}+\frac{c^{2}}{t_{2}^{3}}\right) \right] }},\tag{28}
\end{equation}%
\bigskip

where $T_1=t_1^1+t_1^2+t_1^3, T_2=t_2^1+t_2^2+t_2^3$, and for $i=1,2,$ we have%
\begin{equation*}
\frac{t_{i}^{1}t_{i}^{2}t_{i}^{3}}{(T_i)^{2}}\left( \frac{a^{2}}{t_{i}^{1}}+%
\frac{b^{2}}{t_{i}^{2}}+\frac{c^{2}}{t_{i}^{3}}\right) =
\frac{\underset{cyc}{\prod }a^{k_{i}}(s-a)^{l_{i}}(b+c)^{m_{i}}}{\underset{%
cyc}{\sum }a^{k_{i}}(s-a)^{l_{i}}(b+c)^{m_{i}}}\underset{cyc}{\sum }\frac{%
a^{2}}{a^{k_{i}}(s-a)^{l_{i}}(b+c)^{m_{i}}},
\end{equation*}%
and%
\begin{equation*}
\alpha =\frac{a^{k_{1}}(s-a)^{l_{1}}(b+c)^{m_{1}}}{\underset{cyc}{\sum }%
a^{k_{1}}(s-a)^{l_{1}}(b+c)^{m_{1}}}-\frac{%
a^{k_{2}}(s-a)^{l_{2}}(b+c)^{m_{2}}}{\underset{cyc}{\sum }%
a^{k_{2}}(s-a)^{l_{2}}(b+c)^{m_{2}}},
\end{equation*}%
\begin{equation*}
\beta =\frac{b^{k_{1}}(s-b)^{l_{1}}(a+c)^{m_{1}}}{\underset{cyc}{\sum }%
a^{k_{1}}(s-a)^{l_{1}}(b+c)^{m_{1}}}-\frac{%
b^{k_{2}}(s-b)^{l_{2}}(a+c)^{m_{2}}}{\underset{cyc}{\sum }%
a^{k_{2}}(s-a)^{l_{2}}(b+c)^{m_{2}}},
\end{equation*}%
\begin{equation*}
\gamma =\frac{c^{k_{1}}(s-c)^{l_{1}}(a+b)^{m_{1}}}{\underset{cyc}{\sum }%
a^{k_{1}}(s-a)^{l_{1}}(b+c)^{m_{1}}}-\frac{%
c^{k_{2}}(s-c)^{l_{2}}(a+b)^{m_{2}}}{\underset{cyc}{\sum }%
a^{k_{2}}(s-a)^{l_{2}}(b+c)^{m_{2}}}.
\end{equation*}

If $I_1, I_2$ are Cevian points of rank $k_1, k_2$, then formula
(28) becomes
\begin{equation}
\cos \widehat{I_{1}OI_{2}}=\frac
{2R^2-(abc)^{k_1}\frac{S_{2-k_1}}{(S_{k_1})^2}-(abc)^{k_2}\frac{S_{2-k_2}}{(S_{k_2})^2}
+\sum_{cyc}(\frac{b^{k_1}}{S_{k_1}}-\frac{b^{k_2}}{S_{k_2}})(\frac{c^{k_1}}{S_{k_1}}-\frac{c^{k_2}}{S_{k_2}})a^2}
{2\sqrt{[R^2-(abc)^{k_1}\frac{S_{2-k_1}}{(S_{k_1})^2}][R^2-(abc)^{k_2}\frac{S_{2-k_2}}{(S_{k_2})^2}]}},\tag{29}
\end{equation}
where $S_l=a^l+b^l+c^l$.

Here are few special cases of formula (29). For $k_1=0$ and $k_2=1$
we get the centroid $G$ and the incenter $I$ of barycentric
coordinates $(1;1;1)$ and $(a;b;c)$, respectively. Formula (29)
becomes

\begin{equation}
\cos \widehat{GOI}=\frac{6R^{2}-s^{2}-r^{2}+2Rr}{2\sqrt{%
9R^{2}-2s^{2}+2r^{2}+8Rr}\cdot \sqrt{R^{2}-2Rr}},  \tag{30}
\end{equation}%
where $abc=4sRr,$ $S_{0}=3,$ $S_{1}=2s,$ $S_{2}=2(s^{2}-r^{2}-4Rr).$

For $k_{2}=2$ we obtain the Lemoine point $L$ of triangle $ABC$, of barycentric coordinates $%
(a^{2};b^{2};c^{2})$, and other two formulas are generated
\begin{equation}
\cos \widehat{GOL}=\frac{6R^{2}S_{2}-S_{2}^{2}+4S_{4}}{2\sqrt{9R^{2}-S_{2}}%
\cdot \sqrt{R^{2}-S_{2}^{2}-48(Rrs)^{2}}},  \tag{31}
\end{equation}%
where $S_{4}=S_{2}^{2}-2[(s^{2}+r^{2}+4Rr)^{2}-16Rrs^{2}],$ and
\begin{equation}
\cos \widehat{IOL}=\frac{RS_{2}+rS_{2}-4rs^{2}}{2\sqrt{R^{2}-2Rr}\cdot \sqrt{%
S_{2}^{2}-48r^{2}s^{2}}}.  \tag{32}
\end{equation}%
Each of the formulas (30),(31),(32) generates a Blundon type
inequality, but these inequalities have not nice geometric
interpretations.

Let $I_{1},I_{2},I_{3}$ be three Cevian points of rank\textit{\ }$\mathit{%
(k;l;m)}$ with barycentric coordinates as follows:
\begin{equation*}
a^{k_{i}}(s-a)^{l_{i}}(b+c)^{m_{i}}:b^{k_{i}}(s-b)^{l_{i}}(a+c)^{m_{i}}:c^{k_{i}}(s-c)^{l_{i}}(a+b)^{m_{i}},%
\text{ }i=1,2,3,
\end{equation*}%
and let $t_i^1=a^{k_{i}}(s-a)^{l_{i}}(b+c)^{m_{i}}$, $%
t_i^2=b^{k_{i}}(s-b)^{l_{i}}(a+c)^{m_{i}},t_i^3=c^{k_{i}}(s-c)^{l_{i}}(a+b)^{m_{i}}$.
Now, consider the numbers
\begin{equation*}
\alpha _{ij}=\frac{t_{j}^{1}}{t_{j}^{1}+t_{j}^{2}+t_{j}^{3}}-\frac{t_{i}^{1}%
}{t_{i}^{1}+t_{i}^{2}+t_{i}^{3}},
\end{equation*}%
and%
\begin{equation*}
\beta _{ij}=\frac{t_{j}^{2}}{t_{j}^{1}+t_{j}^{2}+t_{j}^{3}}-\frac{t_{i}^{2}}{%
t_{i}^{1}+t_{i}^{2}+t_{i}^{3}},
\end{equation*}%
and%
\begin{equation*}
\gamma _{ij}=\frac{t_{j}^{3}}{t_{j}^{1}+t_{j}^{2}+t_{j}^{3}}-\frac{t_{i}^{3}%
}{t_{i}^{1}+t_{i}^{2}+t_{i}^{3}},
\end{equation*}%
for all $i,j\in \{1,2,3\}.$ Applying the relation (5) we obtain%
\begin{equation*}
I_{i}I_{j}^{2}=-\alpha _{ij}\cdot \beta _{ij}\cdot \gamma _{ij}\cdot
\left( \frac{a^{2}}{\alpha _{ij}}+\frac{b^{2}}{\beta
_{ij}}+\frac{c^{2}}{\gamma _{ij}}\right),
\end{equation*}%
for all $\ i,j\in \{1,2,3\}.$ Using the Cosine Law in triangle $%
I_{1}I_{2}I_{3}$ it follows%
\begin{equation*}
\cos \widehat{I_{1}I_{2}I_{3}}=\frac{%
I_{1}I_{2}^{2}+I_{2}I_{3}^{2}-I_{3}I_{1}^{2}}{2I_{1}I_{2}\cdot
I_{2}I_{3}}=
\end{equation*}%
\begin{equation}
\frac{-a^{2}(\beta _{12}\gamma _{12}+\beta _{23}\gamma _{23}-\beta
_{31}\gamma _{31})-b^{2}(\gamma _{12}\alpha _{12}+\gamma _{23}\alpha
_{23}-\gamma _{31}\alpha _{31})+c^{2}(\alpha _{12}\beta _{12}+\alpha
_{23}\beta _{23}-\alpha _{31}\beta _{31})}{2\sqrt{-\beta _{12}\gamma
_{12}a^{2}-\gamma _{12}\alpha _{12}b^{2}-\alpha _{12}\beta
_{12}c^{2}}\cdot \sqrt{-\beta _{23}\gamma _{23}a^{2}-\gamma
_{23}\alpha _{23}b^{2}-\alpha _{23}\beta _{23}c^{2}}} \tag{33}
\end{equation}

\begin{theorem}
The following inequalities hold%
\begin{equation*}
-2\sqrt{-\beta _{12}\gamma _{12}a^{2}-\gamma _{12}\alpha
_{12}b^{2}-\alpha _{12}\beta _{12}c^{2}}\cdot \sqrt{-\beta
_{23}\gamma _{23}a^{2}-\gamma _{23}\alpha _{23}b^{2}-\alpha
_{23}\beta _{23}c^{2}}\leq
\end{equation*}%
\begin{equation*}
-a^{2}(\beta _{12}\gamma _{12}+\beta _{23}\gamma _{23}-\beta
_{31}\gamma _{31})-b^{2}(\gamma _{12}\alpha _{12}+\gamma _{23}\alpha
_{23}-\gamma _{31}\alpha _{31})+c^{2}(\alpha _{12}\beta _{12}+\alpha
_{23}\beta _{23}-\alpha _{31}\beta _{31})\leq
\end{equation*}%
\begin{equation}
2\sqrt{-\beta _{12}\gamma _{12}a^{2}-\gamma _{12}\alpha
_{12}b^{2}-\alpha _{12}\beta _{12}c^{2}}\cdot \sqrt{-\beta
_{23}\gamma _{23}a^{2}-\gamma _{23}\alpha _{23}b^{2}-\alpha
_{23}\beta _{23}c^{2}}  \tag{34}
\end{equation}
\end{theorem}

\begin{proof}
The inequalities (34) are simple direct consequences of the inequalities $%
-1\leq \cos \widehat{I_{1}I_{2}I_{3}}\leq 1.$
\end{proof}

\bigskip

\noindent $\begin{array}{l} \mbox{\it Dorin Andrica}\\
\mbox{"Babe\c{s}-Bolyai" University}\\
\mbox{Faculty of Mathematics and Computer Science}\\
\mbox{Cluj-Napoca, Romania}\\
\mbox{e-mail: dandrica@math.ubbcluj.ro }\\
 and\\
\mbox{ King Saud University}\\
\mbox{College of Science}\\
\mbox{Department of Mathematics}\\
\mbox{Ryiadh, Saudi Arabia}\\
\mbox{e-mail:dandrica@ksu.edu.sa}
\end{array}$

\bigskip

\noindent $\begin{array}{l}\mbox{\it Catalin Barbu}\\
\mbox{"Vasile Alecsandri" National College}\\
\mbox{600011 Bacau, Romania}\\
\mbox{e-mail:kafka{\_}mate@yahoo.com}
\end{array}$

\bigskip

\noindent $\begin{array}{l}\mbox{\it Nicusor Minculete}\\
\mbox{Department of REI}\\
\mbox{Dimitrie Cantemir-University}\\
\mbox{Str. Bisericii Romane, nr.107, Brasov, Romania}\\
\mbox{e-mail:minculeten@yahoo.com}
\end{array}$
\end{document}